\documentclass{amsart}
\usepackage{amssymb}   
\usepackage{amsmath}
\usepackage{amsthm}
\usepackage[all,dvips]{xy}
\usepackage[]{pstricks,pst-node,pst-plot}
\usepackage{graphicx}
\usepackage{pstcol,pst-char}
\def\a={{\buildrel a \over =}}
\def\na={{\buildrel a \over \neq}}
\def\CC{{\mathbb C}}
\def\PP{{\mathbb P}}

\def\QQ{{\mathbb Q}}

\def\C{{\mathcal C}}
\def\H{{\mathcal H}}
\def\M{{\mathcal M}}

\def\bM{\overline{\M}}

\def\Lg{\Lambda}

\newtheorem{theorem}{Theorem}[section]
\newtheorem{lemma}[theorem]{Lemma}
\newtheorem{proposition}[theorem]{Proposition}

\newtheorem{definition-lemma}[theorem]{Definition-Lemma}

\theoremstyle{definition}

\theoremstyle{remark}

\def\vandaag{\number\day\space\ifcase\month\or
 januari\or februari\or  maart\or  april\or mei\or juni\or  juli\or
 augustus\or  september\or  oktober\or november\or  december\or\fi,
\number\year}
\def\today{\ifcase\month\or
 Jan\or Febr\or  Mar\or  Apr\or May\or Jun\or  Jul\or
 Aug\or  Sep\or  Oct\or Nov\or  Dec\or\fi
 \space\number\day, \number\year}

\begin{document}
\title[The Hodge Bundle on Hurwitz Spaces]{The Hodge Bundle on Hurwitz Spaces}
\author{Gerard van der Geer $\,$}
\address{Korteweg-de Vries Instituut, Universiteit van
Amsterdam, Postbus 94248, 1090 GE Amsterdam,  The Netherlands}
\email{G.B.M.vanderGeer@uva.nl}
\author{$\,$ Alexis Kouvidakis}
\address{Department of Mathematics, University of Crete,
GR-71409 Heraklion, Greece}
\email{kouvid@math.uoc.gr}
\begin{abstract}
In 2009 Kokotov, Korotkin and Zograf gave in \cite{KKZ} a formula for the
class of the Hodge bundle on the Hurwitz space of admissible covers of
genus $g$ and degree $d$ of the projective line. They gave an analytic
proof of it. In this note we give an algebraic proof and an
extension of the result.
\end{abstract}
\maketitle
\hfill{\sl In memoriam Eckart Viehweg}
\begin{section}{Introduction}\label{sec:intro}
Let ${\H}_{g,d}$ be the Hurwitz space of degree $d$ covers of ${\PP}^1$
of genus $g$ and with simple branch points. It parametrizes covers 
$f: C \to {\PP}^1$ with $C$ irreducible, smooth of genus $g$ and $f$ of
degree~$d$ with $b= 2g-2+2d$ simple branch points that are marked. 
This space admits a 
compactification $\overline{\H}_{g,d}$, the space of admissible covers
of genus $g$ and degree~$d$; that is, covers $f: C \to P$,
where $C$ is a nodal curve, $P$ is a stable $b$-pointed 
curve of genus~$0$
and $f$ an admissible cover in the sense of \cite{HMu}, see also \cite{HMo}.

The Hurwitz space $\overline{\H}_{g,d}$ is a coarse moduli space, but it is 
not smooth, not even normal. The boundary $\overline{\H}_{g,d}-{\H}_{g,d}$ 
consists of finitely many divisors $\Delta_{k,\mu}=\Delta_{b-k,\mu}$ 
indexed by a partition $b= k+(b-k)$
with $2 \leq k \leq b-2$ and a `format' $\mu=(m_1,\ldots,m_r)$ where the
$m_i$ are natural numbers with $\sum_{i=1}^r m_i=d$. 
Note that $\Delta_{k,\mu}$ will in general be reducible.
A generic point in $\Delta_{k,\mu}$  
corresponds to the case where $P$ is a curve of genus $0$ 
with two components $P_1$ and $P_2$ intersecting
in one point $Q$, with $P_1$ having $k$ branch points, and 
the inverse image of $Q$ consists of $r$ points
$Q_1,\ldots,Q_r$ with ramification indices $m_1,\ldots,m_r$.

There are natural maps $q:\overline{\H}_{g,d}\to \overline{\M}_{0,b}$
to the moduli space $\overline{\M}_{0,b}$ of stable $b$-pointed curves 
of genus $0$, and $\varphi: \overline{\H}_{g,d} \to \overline{\M}_g$
to the moduli space $\overline{\M}_g$ of stable curves of genus $g$.
These are defined by assigning to $f: C \to P$ the stable curve $P$,
resp.\ the stabilized model of $C$.

The Hurwitz space $\overline{\H}_{g,d}$ carries a natural ${\QQ}-$divisor
class, the Hodge class $\lambda$. It is the pullback of the first Chern 
class of the Hodge bundle on $\overline{\M}_g$. The boundary divisors
$\Delta_{k,\mu}$ define ${\QQ}$-divisor classes which are taken in the
orbifold sense, that is, counted with a weight $1/a$ where $a$ is the
order of the automorphism group of an object corresponding to the
generic point of the boundary divisor, see Section \ref{DivHurwitz}.

The theorem of Kokotov, Korotkin and Zograf expresses the Hodge class
$\lambda$ on $\overline{\H}_{g,d}$ in terms of the boundary divisor 
classes. In terms of the Picard group of the corresponding functor
the theorem reads

\begin{theorem}\label{KKZThm}
The Hodge class $\lambda$ of the functor $\overline{\H}_{g,d}$ is given by
$$
\lambda = \sum_{k=2}^{b/2} m(\mu) \, \left( \frac{k(b-k)}{8(b-1)}
-\frac{1}{12}(d - \sum_{i=1}^r \frac{1}{m_i})\right) \delta_{k,\mu},
$$
where $b=2g+2d-2$ and $\delta_{k,\mu}$ is the class of $\Delta_{k,\mu}$
and $m(\mu)$ is the least common multiple of $m_1,\ldots,m_r$.
\end{theorem}

Kokotov, Korotkin and Zograf used anaytic tools, 
esp.\ the tau-function,
to construct a trivializing section of the Hodge bundle on ${\H}_{g,d}$
and calculated the vanishing orders of this section along the boundary
divisors.
In our approach we shall apply Grothendieck-Riemann-Roch to the relative
dualizing sheaf on the admissible cover $f: C\to P$ over a base $S$
with structure map $t: C \to S$; this
will provide us with an expression of $\lambda$ in terms of 
$t_{*}{\omega_t}^2$ and boundary divisors and then we will express
$t_{*}{\omega_t}^2$ too in terms of boundary divisors.
In fact, if $f:C/S\to P/S$ is any family of admissible covers of genus $g$ 
and degree $d$ then we have a commutative diagram
\begin{displaymath}
\begin{xy}
\xymatrix{
{\C} \ar[r]^{f} \ar[d]^{t} & \bM_{0,b+1}\ar[d]^{\pi_{b+1}} \\
S \ar[r]^{q} & \bM_{0,b} \\
}
\end{xy}
\end{displaymath}
The relative dualizing sheaf $\omega_t$ of $t$ can be expressed in terms
of the ramification divisor $R$ of the map $f$ and the pullback
of the relative dualizing sheaf $\omega_{\pi_{b+1}}$ of $\pi_{b+1}$.
Working this out will give us our proof.

This proof extends also to other cases. 
As an example we consider the Hurwitz space ${\H}_{g,d,l}$ of 
covers $f: C \to P$ with $C$ smooth of genus $g$ and
$P=({\PP}^1,p_1,\ldots,p_b)$ a $b$-pointed smooth curve of genus $0$, 
$f$ a morphism of degree $d$ 
with simple branch points $p_i$ for $i=2,\ldots,b$ 
and one branch point $p_1$, over which $f$ is \'etale
except for one $l$-fold ramification point. Note that $2g=b-2d+l$.
The dimension of this space is $2g+2d-l-3$.
This space admits a compactification $\overline{\H}_{g,d,l}$ by
so-called $l$-admissible covers where the curve $f:C \to P$ satisfies
the axioms for admissibility except over $p_1$ over which we have exactly
one ramification point of degree $l$, cf.\ \cite{H}, where this 
Hurwitz space was introduced.
Then the expression for $\lambda$ is the following.

\begin{theorem}\label{ExtThm}
The Hodge class of the functor $\overline{\H}_{g,d,l}$  is given by
$$
\begin{aligned}
\lambda = & \sum_{k=2}^{[b/2]} 
\sum_{\mu} m(\mu) \, \left( \frac{k(b-k)}{8(b-1)}
-\frac{1}{12}(d - \sum_{i=1}^r \frac{1}{m_i})\right) \delta_{k,\mu} \\
& \qquad + \frac{(2l+1)(l-2)}{24 \, l}\sum_{k=2}^{b-2} 
\sum_{\mu} m(\mu)\,  \frac{(b-k)(b-1-k)}{(b-1)(b-2)}
\delta_{k,\mu}^{1} \, ,\\
\end{aligned}
$$
where $b=2g+2d-l$ and $\delta_{k,\mu}^1$ the part of $\delta_{k,\mu}$
with general member an admissible cover that maps to a $b$-pointed curve
whose component with the $k$ marked points contains the point $p_1$.
\end{theorem}
In \cite{GK} we gave an application of the formula of Kokotov, Korotkin 
and Zograf by calculating an important divisor class on ${\bM}_g$ for
$g$ even.
\end{section}
\begin{section}{Divisors on $\overline{\M}_{0,b}$}
We recall some basic facts about the divisor theory of
$\bM_{0,b}$, see \cite{Ke}. The boundary of $\bM_{0,b}$ is
the union of  irreducible divisors,
each of which corresponds to a decomposition of $B=\{1,\ldots,b\}$
as $B=\Lambda \sqcup \Lambda^c$ into two disjoint
subsets with $2\leq\# \Lg  \leq b-2$.
We write the corresponding divisor as $S_b^{\Lg}$
modulo the relation $S_b^{\Lg}=S_b^{{\Lg}^c}$.
If one wishes one can normalize the $\Lambda$ by requiring that
$$
\# (\Lambda \cap \{1,2,3\}) \leq 1.
$$
With  $\Lg \subset \{1,\ldots, b\}$, the generic element of
the divisor $S_b^\Lg$ represents a stable curve with two rational
components, with the marked points of $\Lambda$ on one component.
The map $\pi_{b+1}: \bM_{0,b+1} \to \bM_{0,b}$ is
equipped with $b$ sections $s_i: \bM_{0,b} \to \bM_{0,b+1}$
with $i=1,\ldots,b$. We can interpret $\overline{M}_{0,b+1}$
as the universal curve over $\overline{M}_{0,b}$.

The boundary divisors of $\bM_{0,b+1}$ are related to those of
$\bM_{0,b}$ as follows:
$$
\pi_{b+1}^*S_b^{\Lg}= S_{b+1}^{\Lg} \cup S_{b+1}^{\Lg \cup \{ b+1 \} },
\eqno(1)
$$
with $\Lg \subset \{1, \ldots, b \}$.
Note that if $\Lambda \subset \{1,\ldots,b\}$ is normalized,
then so are $\Lambda$ and $\Lambda\cup \{ b+1 \}$ as subsets of
$\{1,\ldots,b+1\}$.
So all the boundary components of
$\bM_{0,b+1}$ are coming from $\bM_{0,b}$ except the components
$S_{b+1}^{ \{i, b+1\}}$ ($i=1, \ldots, b$) that correspond to the
image of the $b$ sections $s_i$.
The map $S_{b+1}^\Lg  \to S_b^\Lg$
(resp.\ $S_{b+1}^{\Lg \cup\{b+1\}} \to S_b^\Lg$) is generically a
$\PP ^1$-fibration.

Recall that $B=\{1,\ldots,b\}$.
In $\bM_{0,b}$ we define for $2 \leq j < b/2$
the divisors
$$
T^j_b=\sum_{\Lambda \subset B, \, \# \Lambda =j} S_b^{\Lambda}
\quad
\hbox{\rm and if $b$ is even} \quad
T^{b/2}_b={1\over 2} 
\sum_{\Lambda \subset B, \, \# \Lambda=b/2} S_b^{\Lambda} .
$$
On the moduli space $\overline{\M}_{0,b}$ 
we have the tautological classes $\psi_i$ for $1\leq i \leq b$,
defined as the first Chern class of the line bundle that
associates to a pointed curve $(C,p_1,\ldots,p_b)$ the
cotangent space to $C$ at $p_i$. We put $\psi= \sum_{i=1}^b \psi_i$.

On $\overline{\M}_{0,b}$ we have the relation
$$
\psi_i= \sum_{j=1}^{b-3} \frac{(b-1-j)(b-2-j)}{(b-1)(b-2)}
\sum_{A \subset \{1,\ldots,b\}-\{i\}, \#A=j} S^{\{i\} \cup A}_b,
$$
cf.\ \cite{FG}, Lemma 1 on page 1186. Therefore we get
$$
\psi:=\sum_{i=1}^b \psi_i = \sum_{j=2}^{[b/2]} \frac{(b-j)j}{b-1} T^j_b \, .
\eqno(2)
$$
\end{section}
\begin{section}{Divisors on the Hurwitz Space}\label{DivHurwitz}
In this section we are concerned with divisors on our Hurwitz space.
Recall that the boundary of $\overline{\H}_{g,d}$ consists of finitely
many divisors $\Delta_{k,\mu}=\Delta_{b-k,\mu}$, 
where for each irreducible component of
$\Delta_{k,\mu}$ the generic point corresponds to an admissible cover
$f: C \to P$ of degree $d$ 
with $P$ a stable $b$-pointed genus $0$ curve consisting of
two copies $P_i$ ($i=1,2$) of ${\PP}^1$ with one intersection point $Q$ 
and with $k$ marked points on $P_1$ and
with ramification points $Q_i$ of ramification degree $m_i$ ($i=1,\ldots,r$)
over $Q$. One can decompose these divisors further as $\Delta_{k,\mu}=
\sum \Delta_{\Lambda,\mu}$ corresponding to decompositions
$\{1,\ldots,b\}=\Lambda \sqcup \Lambda^c$.

We start by giving a local description of
admissible covers, cf.\ the discussion in Harris-Mumford \cite{HMu}, p.\
61--62. We then have the diagram
\begin{displaymath}
\begin{xy}
\xymatrix{
& \overline{\M}_{0,b+1} \ar[d]^{\pi_{b+1}}\\
\overline{H}_{g,d} \ar[r]^{q} & \overline{\M}_{0,b}
}
\end{xy}
\end{displaymath}
We now take a general point  $\gamma$ of an irreducible component
of $\Delta_{k,\mu}$.
It corresponds to a general admissible cover $f: C \to P$ with $P$
a stable genus $0$ curve with two components $P_1$ and $P_2$ intersecting
transversally in one point $Q$. The point $\gamma$ maps under $q$ to the 
point of $S_{b}^{\Lambda}$ determined by $P$. 
If $Q_1,\ldots,Q_r$ are the pre-images
of $Q$ under $f$ with ramification format $\mu=(m_1,\ldots,m_r)$
then locally near $Q_i$ the curve $C$ is given by $x_iy_i=s_i$ and locally
near $Q$ the curve $P$ is given 
by $uv=t_1$ with $u=x_i^{m_i}$ and $v=y_i^{m_i}$.
By \cite{HMu} 
the formal neighborhood of $\gamma$ that pro-represents infinitesimal
deformations is given by ${\rm Spec}(R)$ with
$$
R={\CC}[t_1,t_2,\ldots,t_{b-3},s_1,\ldots,s_r]/
\langle s_j^{m_j}=t_1, \, j=1,\ldots,r\rangle \, .
$$
We will use the notation
$$
m(\mu)= \text{least common multiple of } \, m_1,\ldots,m_r \, .
$$

The following fact is known; for the reader's convenience we give a proof.

\begin{lemma}\label{normalization}
The normalization of $\overline{\H}_{g,d}$ has $m_1\cdots m_r/m(\mu)$ 
branches $E_i$ along any irreducible component of $\Delta_{k,\mu}$ 
and the ramification degree of the map $E_i \to \overline{\M}_{0,b}$ is
$m(\mu)$.
\end{lemma}
\begin{proof}
To normalize we introduce the local parameter $\tau=t_1^{1/m(\mu)}$ and write
$s_j=\zeta_{m_j}\tau^{m(\mu)/m_j}$ with $\zeta_{m_j}$ an $m_j$th root of unity. 
This will give a normalization. By letting the 
$m(\mu)$th roots of unity act via reparametrizations
$\tau \mapsto \zeta_{m(\mu)} \tau$
we see that $\mu(m)$ solutions give the same branch, leading to
$m_1\cdots m_r/m(\mu)$ branches. 
For each of the branches the ramification index is $m(\mu)$.
\end{proof}

We want to prove a relation in the rational Picard group
${\rm Pic}_{\QQ}(\overline{H}_{g,d})$. Unfortunately, this
group is not known, but it is conjectured that it is
generated by boundary classes, cf.\ \cite{HMo}, p.\ 66, Conj.\ 2.49.

We let $\lambda$ be the Hodge class on $\overline{\H}_{g,d}$. It is defined
functorially by taking for each admissible cover
\begin{displaymath}
\begin{xy}
\xymatrix{
 C \ar[rd]_{\phi} \ar[r]^{f} & P \ar[d]^{}\\
& S \\
}
\end{xy}
\end{displaymath}
the first Chern class of $\phi_* \omega_{C/S}$ with $\omega_{C/S}$
the relative dualizing sheaf.
 
The divisors that we use will be viewed in the orbifold sense;
that is, we will weight an irreducible divisor $D$  with a factor $1/A$
where $A$ is the order of the automorphism group of the object 
corresponding to the generic point of $D$. 
So in a relation 
$\lambda= \sum c_{k,\mu} \delta_{k,\mu}$, like the one of Theorem
\ref{KKZThm}, the class
$\delta_{k,\mu}$ is the sum of the classes of the irreducible components 
of $\Delta_{k,\mu}$ weighted by $1$ over the order of the
automorphism group of the
object corresponding to the generic point of the component.

In order to prove the relation we first work on the normalization 
$\tilde{\H}_{g,d}$ of $\overline{\H}_{g,d}$.
To prove a relation 
$\tilde{\lambda}= \sum \tilde{c}_{k,\mu} \tilde{\delta}_{k,\mu}$
there it suffices to prove 
for every family of admissible covers
over a smooth $1$-dimensional base $S$ 
not contained in the boundary of our Hurwitz space, a relation 
$\lambda_S= \sum \tilde{c}_{k,\mu} \delta_{k,\mu}^{\prime}$ 
with $\lambda_S$ and $\delta^{\prime}_{k,\mu}$ 
the pullbacks under the classifying maps
$S \to \tilde{H}_{g,d}$, cf.\ the discussion on p.\ 141-146
of \cite{HMo}. Then if we descend from the normalization 
$\tilde{\H}_{g,d}$ of $\overline{\H}_{g,d}$
to the space $\overline{\H}_{g,d}$ itself by the normalization map
$\nu: \tilde{\H}_{g,d} \to \overline{\H}_{g,d}$
 we have to take into account that the pushforward of a boundary component 
$\tilde{\Delta}_{k,\mu}$ on $\tilde{\H}_{g,d}$  results in 
a multiple (cf.\ Lemma \ref{normalization})
of the (reduced) cycle $\Delta_{k,\mu}$ on $\overline{\H}_{g,d}$, viz.\
$$
\nu_*[\tilde{\Delta}_{k,\mu}]= 
\frac{m_1\cdots m_r}{m(\mu)} [\Delta_{k,\mu}] \,  \, ,
$$
and this gives rise to a factor 
$m_1\cdots m_r/m(\mu)$ in the coefficients $c_{k,\mu}$.

Using now $1$-dimensional smooth families of admissible curves
we now prove that $\lambda$ is a pullback from ${\bM}_g$.

\begin{lemma}\label{lambda=lambda}
The Hodge bundle on $\overline{\H}_{g,d}$ is the pullback under 
the natural map
$\varphi: \overline{\H}_{g,d} \to {\bM}_g$ of the Hodge bundle on
${\bM}_g$.
\end{lemma}
\begin{proof} 
Given such a family $f: C \to P/S$ of admissible covers
we desingularize the total space of $C$ to get $f^{\prime}: C^{\prime}\to S$.
This does not affect the Hodge bundle, see e.g.\
\cite{HMo}, p.\ 156. Then we stabilize  $C^{\prime}$ to get 
a family $X$ of stable curves over $S$. This is done step by step
by first contracting exceptional curves
in fibres of $f^{\prime}$, say $\alpha : C^{\prime}\to X$ is a contraction. 
But this does not change the Hodge bundle:  one has $\alpha^* \lambda_X=
 \lambda_{C^{\prime}}$; indeed, the degree of $\omega_{C^{\prime}}$
on an exceptional curve is negative, hence the pullback induces an 
isomorphism $H^0(X_s,\omega_X) \cong H^0(C^{\prime}_s,\omega_{C^{\prime}})$.
In this way we arrive at a model $X^{\prime}$ over $S$; we then have to
contract the $(-2)$-configurations in the fibres to get the family
of stable curves $X/S$. As remarked above this does not change 
the relative dualizing sheaf: if $\beta: X^{\prime} \to X$ is the contraction
then $\beta^* \omega_X = \omega_X^{\prime}$. 
\end{proof}
\end{section}
\begin{section}{The Proof}
Let $f:C/S \to P/S$ be any family of admissible covers of genus $g$ 
and degree $d$. Since $\overline{\M}_{0,b+1}$ is the universal 
curve over $\overline{\M}_{0,b}$ we have a commutative diagram 
\begin{displaymath}
\begin{xy}
\xymatrix{
&{\C} \ar[r]^{f} \ar[d]^{t} & \bM_{0,b+1}\ar[d]^{\pi_{b+1}} \\
\overline{\H}_{g,d} & S \ar[l]_{h} \ar[r]^{q} & \bM_{0,b} \\
} 
\end{xy} \eqno(3)
\end{displaymath}
Interpreting the boundary divisor and taking into account the
ramification degree of the normalization (see Lemma \ref{normalization} and
its proof) we obtain the following.
\begin{lemma}\label{qstarT}
We have $q^*(T_b^j)= \sum_{\mu} m(\mu)\, h^* \delta_{j,\mu}$.
\end{lemma}

Hereafter we just shall write $\delta_{j,\mu}$ 
for $h^*(\delta_{j,\mu})$ on $S$. This is the divisor class on $S$ given
by the singular curves of type $(j,\mu)$.

The families $C/S$ and $P/S$ have relative dualizing line bundles
$\omega_t$ and $\omega_{\pi_{b+1}}$. 
We let $R$ be the closure of the ramification locus of the map
$f$ restricted to the open part of $S$ which is the pullback under
$S \to {\H}_{g,d}$. Under $f$ this maps to the $b$ sections of the
map $\pi_{b+1}$. In view of the diagram (3) we have
$$
\omega_t=f^*\omega_{\pi_{b+1}}+R \, .
$$
Now $R$ splits as a sum $\sum_{i=1}^b R_i$ with $R_i$ the component of $R$
mapping to the $i$th section $S_{b+1}^{\{i,b+1\}}$. Since the branching 
is simple we thus obtain $b$ sections $\tau_i: S \to {\C}$ of the map $t$.
\begin{lemma}\label{Rsquare}
We have
$t_*(f^*\omega_{\pi_{b+1}}\cdot R)=q^* \psi$ and
$t_*(R^2)=-(1/2) q^* \psi$.
\end{lemma}
\begin{proof}
We calculate
$$
\begin{aligned}
t_*(f^*\omega_{\pi_{b+1}} \cdot R)  & =
\sum_{i=1}^b t_*(f^* \omega_{\pi_{b+1}} \cdot R_i)= 
\sum_{i=1}^b t_*(\tau_{i*}\tau_i^{*}f^*\omega_{\pi_{b+1}})\\
& =\sum_{i=1}^b q^* s_i^* \omega_{\pi_{b+1}}
=q^*(\sum_{i=1}^b \psi_i)=q^*\psi\, , \\
\end{aligned}
$$
where we used $f \circ \tau_i=s_i \circ q$ and $t \circ \tau_i={\rm id}_S$.
This proves the first claim.

In order to calculate $t_*(R^2)$ we observe that the $R_i$ are disjoint,
hence $R^2=\sum R_i^2$. If we write
$f^* S_{b+1}^{\{i,b+1\}}=2R_i+A_i$ with $A_i\cdot R_i=0$, we have
$$
R_i^2=\frac{1}{2}(f^* S_{b+1}^{\{i,b+1\}} -A_i)\cdot R_i =
\frac{1}{2} \, f^* S_{b+1}^{\{i,b+1\}}\cdot R_i \, .
$$
Now we use the adjunction formula to observe that
$s_i^* S_{b+1}^{\{i,b+1\}}= -\psi_i$.
So we see
$$
\begin{aligned}
t_*(R^2) &=
\frac{1}{2} \sum_{i=1}^b t_*(f^*S_{b+1}^{\{i,b+1\}} \cdot R_i) =
\frac{1}{2} \sum_{i=1}^b t_*\tau_{i*}\tau_i^* f^* S_{b+1}^{\{i,b+1\}}\\
&=\frac{1}{2} \sum_{i=1}^b q^*s_i^* S_{b+1}^{\{i,b+1\}} =
- \frac{1}{2} q^*\sum_{i=1}^b \psi_i= - \frac{1}{2} q^* \psi \, .
\end{aligned}
$$
\end{proof}
\begin{lemma}\label{omegapisquare}
We have $t_*f^*(\omega_{\pi_{b+1}}^2)= -d \, \sum_{j=2}^{b/2} q^* T_b^j$.
\end{lemma}
\begin{proof}
We write  $\hat{S}= \sum_{i=1}^b S_{b+1}^{\{i,b+1\}}$ for the sum of the sections.
Recall from \cite{FG} (Section 2) the identity
$$
\pi_{b+1*}([\omega_{\pi_{b+1}}(\hat{S})]^2)=\kappa_1(\overline{\M}_{0,b})=
\psi-\sum T_b^j \, .
$$
If $\kappa=\omega_{\pi_{b+1}} \cdot \hat{S}$ then $\pi_{b+1*}(\kappa)=\psi$.
Indeed, $\pi_{b+1*}(\kappa)=\sum_{i}\pi_{b+1*}s_{i*}s_i^*\omega_{\pi_{b+1}}=
\sum_i s_i^*\omega_{\pi_{b+1}}=\sum_i\psi_i=\psi$.
Moreover, by the adjunction formula we have $\pi_{b+1*}(\hat{S}^2)=-\psi$.
Therefore we get
$$
\psi-\sum_{j=2}^{b/2} T_b^j = \pi_{b+1*}(\kappa^2)=
\pi_{b+1*}(\omega_{\pi_{b+1}}^2+ 
2 \omega_{\pi_{b+1}}\cdot \hat{S} + \hat{S}^2) 
= \pi_{b+1*}[\omega_{\pi_{b+1}}^2]  +\psi \, ,
$$
hence 
$\pi_{b+1*}[\omega_{\pi_{b+1}}^2]= -\sum_{j=2}^{b/2} T_b^j$.
Now using
$t_*f^*[\omega_{\pi_{b+1}}^2]=d\, q^*\pi_{b+1*}[\omega_{\pi_{b+1}}^2]$
we get the required result.
\end{proof}
\begin{proposition}\label{omega-t-square}
We have the identity
$$
t_*[\omega_t^2]= 
\sum_{j=2}^{b/2} \left( -d+ \frac{3}{2} \frac{j(b-j)}{b-1}\right)
\sum_{\mu} m(\mu) \delta_{j,\mu}.
$$
\end{proposition}
\begin{proof}
We have $t_*[\omega_t^2]=t_*(f^*\omega_{\pi_{b+1}}^2+2f^*\omega_{\pi_{b+1}}
\cdot R +R^2)$ and the result now follows from lemmas \ref{Rsquare} and 
\ref{omegapisquare} and the formula for $\psi$ given in (2).
\end{proof}

We now prove Theorem \ref{KKZThm}.
In \cite{MStab} Mumford proved the identity $12 \lambda = \kappa_1 +\delta$ 
on the moduli space $\overline{\M}_g$. The same approach,
applying Grothendieck-Riemann-Roch to the morphism $t$ and the relative 
dualizing sheaf, works for 
our admissible curve $C/S$ and the proof of Theorem 5.10 
of loc.\ cit.\ can be transferred almost 
verbatim (taking into account the arguments
in the proof of Lemma \ref{lambda=lambda}) to give the identity 
$$
12 \lambda = t_*[\omega_t^2] + \sum_{j,\mu}m(\mu) \sum_{i=1}^r 
\frac{1}{m_i} \delta_{j,\mu}.
\eqno(4)
$$
on $S$. Indeed, to get the multiplicity of $\delta_{j,\mu}$ we observe 
the following.
Let $f: C/S \to P/S$ be a family of admissible covers over a 
smooth $1$-dimensional
and such that the general fibre is a smooth curve.
We may assume that the image of $S$ in $\overline{\H}_{g,d}$
intersects the divisors $\Delta_{j,\mu}$ only in sufficiently general points
(omitting codimension $2$ loci where boundary divisors intersect). 
Let $s \in S$ correspond to an admissible
cover $C_s \to P_s$ with $P_s$ consisting of two components meeting in $Q$
with local equation $uv=t_1$
and $Q_1,\ldots,Q_r$ the points lying over $Q$ with ramification format
$\mu=(m_1,\ldots,m_r)$ and with local equations $x_iy_i=s_i$ and $s_i^{m_i}=t_1$
and $u=x_i^{m_i}$, $v=y_i^{m_i}$. Then when we pull back to $S$ we have
because the normalization map on $\overline{\H}_{g,d}$ is ramified of degree
$m(\mu)$, a change of coordinates $t_1=\sigma^{m(\mu)}$ which leads to
local equations around $Q_i$ of the form $x_iy_i=\sigma^{m(\mu) /m_i}$.
Therefore the contribution of the $Q_i$'s to the singularity locus is
$m(\mu) \sum_{i=1}^r 1/m_i$.
Substituting the expression for $t_*[\omega_t^2]$ from Proposition
\ref{omega-t-square} in equation (4) yields the result.

The proof of Theorem \ref{ExtThm} goes in a similar fashion.
We consider a diagram
\begin{displaymath}
\begin{xy}
\xymatrix{
&{\C} \ar[r]^{f} \ar[d]^{t} & \bM_{0,b+1}\ar[d]^{\pi_{b+1}} \\
\overline{\H}_{g,d,l} & S \ar[l]_{h} \ar[r]^{q} & \bM_{0,b} \\
}
\end{xy}
\end{displaymath}
The analogue of Lemma \ref{qstarT} holds and we have the formula 
$$
\omega_t= f^*\omega_{\pi_{b+1}}+(l-1)R_1+\sum_{i=2}^b R_i=
 f^*\omega_{\pi_{b+1}}+ (l-2)R_1+R\, .
$$
The analogue of Lemma \ref{Rsquare} says that
$t_*(f^*\omega_{\pi_{b+1}}\cdot R_i)=q^* \psi_i$ and
$t_*(R^2)=-(1/l)q^*\psi_1-(1/2) q^* \sum_{i=2}^b\psi_i$.
Indeed, from $f^* S_{b+1}^{\{1,b+1\}}= lR_1 +A_1$ with $A_1\cdot R_1=0$
we deduce
$$
R_1^2=\frac{1}{l}(f^*S_{b+1}^{\{1,b+1\}} -A_1)\cdot R_1
= \frac{1}{l}f^*S_{b+1}^{\{1,b+1\}} \cdot R_1 \, ,
$$
while for $i=2,\ldots,b$ we find as above 
$R_i^2= \frac{1}{2} f^* S_{b+1}^{\{i,b+1\}} \cdot R_i$,
hence
$$
t_*(R^2)=\sum_{i=1}^b t_*(R_i^2) =
-\frac{1}{l}\psi_1-\frac{1}{2}\sum_{i=2}^b \psi_i \, .
$$
Then we get as formula for $t_*[\omega_t^2]$
$$
t_*[\omega_t^2]= -d \, \sum_{j=2}^{[b/2]} q^* T_b^j + 
\frac{(2l+1)(l-2)}{2l}\,  \psi_1 
+\frac{3}{2}\,  \psi \, .
$$
We now substitute again the formula (4) and work out the pullback of $\psi_1$
that gives us the extra term in the formula for $\lambda$.
This completes the proof of Theorem \ref{ExtThm}.
\end{section}

\end{document}